\documentclass[11pt]{amsart}
\usepackage{amsopn}
\usepackage{amsmath,amsthm,amssymb}

\newcommand{\nc}{\newcommand}

 \nc{\aff}{\mathfrak{aff} } \nc{\bb}{\mathfrak{b} }
\nc{\cc}{\mathfrak{c} }  \nc{\dd}{\mathfrak{d} }
 \nc{\ggo}{\mathfrak{g} }
 \nc{\hh}{\mathfrak{h} }  \nc{\ii}{\mathfrak{i} }
 \nc{\jj}{\mathfrak{j} }  \nc{\kk}{\mathfrak{k} }
\nc{\mm}{\mathfrak{m} }   \nc{\nn}{\mathfrak{n} }
\nc{\pp}{\mathfrak{p} }  \nc{\rr}{\mathfrak{r} } \nc{\sg}{\mathfrak{s} }
 \nc{\sog}{\mathfrak{so} }  \nc{\spg}{\mathfrak{sp} }
 \nc{\sug}{\mathfrak{su} }  \nc{\slg}{\mathfrak{sl} }
 \nc{\tg}{\mathfrak{t} }  \nc{\uu}{\mathfrak{u} }
 \nc{\vv}{\mathfrak{v} } \nc{\ww}{\mathfrak{w} }
 \nc{\zz}{\mathfrak{z} }

 \nc{\ggob}{\overline{\mathfrak{g}}}

\nc{\glg}{\mathfrak{gl} }

\nc{\pca}{\mathcal{P}} \nc{\nca}{\mathcal{N}}

 \nc{\vp}{\varphi} \nc{\ddt}{\frac{{\rm d}}{{\rm d}t}}
 \nc{\la}{\langle} \nc{\ra}{\rangle}

 \nc{\SO}{{\sf SO}} \nc{\Spe}{{\sf Sp}} \nc{\Sl}{{\sf Sl}}
 \nc{\SU}{{\sf SU}} \nc{\Or}{{\sf O}} \nc{\U}{{\sf U}}
 \nc{\Gl}{{\sf Gl}} \nc{\Se}{{\sf S}} \nc{\Cl}{{\sf Cl}}
 \nc{\Spin}{{\sf Spin}} \nc{\Pin}{{\sf Pin}}

 \nc{\RR}{{\mathbb R}} \nc{\HH}{{\mathbb H}} \nc{\CC}{{\mathbb C}}
 \nc{\ZZ}{{\mathbb Z}} \nc{\FF}{{\mathbb F}} \nc{\NN}{{\mathbb N}}
 \nc{\GG}{{\mathbb G}} \nc{\JJ}{{\mathbb J}} \nc{\II}{{\mathbb I}}
 \nc{\KK}{{\mathbb K}} \nc{\DD}{{\mathbb D}}

 \nc{\ad}{\operatorname{ad}} \nc{\Ad}{\operatorname{Ad}}
 \nc{\coad}{\operatorname{coad}} \nc{\ct}{\operatorname{T}}
 \nc{\rank}{\operatorname{rank}} \nc{\Irr}{\operatorname{Irr}}
 \nc{\End}{\operatorname{End}} \nc{\Aut}{\operatorname{Aut}}
 \nc{\Inn}{\operatorname{Inn}} \nc{\Der}{\operatorname{Der}}
 \nc{\Dera}{\operatorname{Dera}} \nc{\Auto}{\operatorname{Auto}}
 \nc{\GL}{\operatorname{GL}}
 \nc{\SL}{\operatorname{SL}}
 \theoremstyle{plain}
 \newtheorem{thm}{Theorem}[section]
 
 \newtheorem{cor}[thm]{Corollary}
 \newtheorem{lem}[thm]{Lemma}

 \theoremstyle{remark}
 \newtheorem*{remark}{Remark}

 \newtheorem{example}[thm]{Example}

 \newcommand{\R}{\mathbb R}

\newcommand{\mg}{\mathfrak n }
\newcommand{\mz}{\mathfrak z }

\newcommand{\mgg}{\mathfrak g }

\newcommand{\lra}{\longrightarrow}

\begin{document}

\title[Symplectic structures on free nilpotent Lie algebras]
{Symplectic structures on free nilpotent Lie algebras}

\author{Viviana J. del Barco}
\address{ECEN-FCEIA, Universidad Nacional de Rosario, Pellegrini 250, 2000 Rosario, Argentina.
}

\email{V. del Barco: delbarc@fceia.unr.edu.ar}

%    author two information

\thanks{Supported by a fellowship from CONICET and research grants form SCyT-UNR, Secyt-UNC}
%\date{\today}

\dedicatory{}

%    "Communicated by" -- provide editor's name; required.
\commby{}

%    Abstract is required.

\begin{abstract} In this work we study the problem of existence of symplectic structures on free nilpotent Lie algebras. Necessary and sufficient conditions are given for even dimensional ones. The one dimensional central extension for odd dimensional free nilpotent Lie algebras is also considered.
\end{abstract}

\maketitle

\section{Introduction}

%Given a nilpotent Lie algebra $\mgg$, a symplectic structure is a skew symmetric bilinear form $\omega$ that is also closed by the Lie algebra differential. Consider $G$ the connected simply connected nilpotent Lie group corresponding to $\mgg$. Then, such an structure is an closed 2-form on $N$ invariant by the group action. 

A symplectic structure on a $2n$-dimensional differentiable manifold $M$ is a closed 2-form $\omega$ such that $\omega^n$ is non singular. %In particular, the cohomology class of $\omega$ is non trivial if $M$ is compact.
In the particular case of a nilmanifold $M=\Gamma\backslash G$, that is, a compact quotient of a nilpotent simply connected Lie group by a cocompact discrete subgroup $\Gamma$, the natural map from $H^i (\mgg)$, $\mgg$ the Lie algebra of $G$, to the de Rham cohomology group $H^i(M,\R)$ is an isomorphism, $0 \leq i \leq 2n$ by Nomizu's work (\cite{NO}). Thus, any symplectic form on $M$ is cohomologous to a left invariant form on $G$ and hence it is represented by a closed 2-form on the Lie algebra $\mgg$.

Therefore, the study of existence of symplectic structures on the nilmanifold $\Gamma \backslash G$ reduces to the existence of a closed 2-form on the Lie algebra $\mgg$ such that  $\omega^n\neq 0$.

%Nilmanifolds itselfs are widely used in topology and geometry, usually as they gave different counter-examples in symplectic and complex topology. The first examples of compact symplectic manifolds with no K\"A hler structure were nilmanifolds (see [25], [7] for references)(millionschikov)

One possible approach to this existence problem is to perform a study case by case in those dimensions where nilpotent Lie algebras are classified. For instance, Morosov described all nilpotent Lie algebras up to dimension six in \cite{MO} and his tables were used by Goze and Bouyakoub in \cite{BO-GO} to give the list of all symplectic Lie algebras of dimension $\leq 6$. %The lack of of a complete classification of nilpotent Lie algebras makes this problem quite difficult. For instance, Morosov described all nilpotent Lie algebras up to dimension six in \cite{MO} and his tables were used by Goze and Bouyakoub in \cite{BO-GO} for proving that all nilpotent Lie algebras of dimension four are symplectic. They also  showed those of dimension six admitting symplectic structures and pointed out that it is the first dimension where a non symplectic nilpotent Lie algebra can be found.
Another alternative is to treat the problem separately on subfamilies of nilpotent Lie algebras. For example, in \cite{DO-TI} the authors work with Heisenberg type nilpotent Lie algebras. Moreover, the classification of symplectic filiform Lie algebras, which are Lie algebras $\mgg$ of nilpotency index $k=\dim \mgg-1$, is given in \cite{MI}. Among nilpotent Lie algebras associated to graphs, a complete description of the symplectic ones can be made in terms of the corresponding graph (\cite{PO-TI}). 

Following the second approach, we restrict ourselves to the family of free nilpotent Lie algebras. The goal of this work is to give an explicit description of those Lie algebras in the family that admit symplectic structures. More precisely, we prove the following result.
\bigskip

{\bf Main Theorem.} {\em Let $\mg_{m,k}$ be the free $k$-step nilpotent Lie algebra on $m$ generators. 
\begin{enumerate}
\item If $\dim \mg_{m,k}$ is even then $\mg_{m,k}$ admits symplectic structures if and only if $(m,k)=(3,2)$.
\item If $\,\dim \mg_{m,k}$ is odd, the one dimensional central extension $\R\oplus \mg_{m,k}$ admits symplectic structures if and only if $(m,k)=(2,2)$.
\end{enumerate}}

%Moreover, there are no general sufficient conditions for a nilpotent Lie algebra to admit symplectic structures. So problem is usually treated separately by families using different tools.  In \cite{DO-TI} the authors work with Heisenberg type nilmanifolds. The classification of symplectic filiform Lie algebras, which are Lie algebras $\mgg$ of nilpotency index $k=\dim \mgg-1$, a classification is given in \cite{MI}. Among nilpotent Lie algebras associated to graphs, a complete description of the symplectic ones can be made in terms of the corresponding graph (\cite{PO-TI}).

% The goal of this work is to give a complete answer for the existence of symplectic structures in the family of free nilpotent Lie algebras. 
\smallskip

In \cite{BE-GO} Benson and Gordon give a necessary condition for the existence of symplectic structures on nilpotent Lie algebras. We develop that condition for free nilpotent Lie algebras using the Hall basis, throughout which we obtain specific restriction in that family enabling the proof of the Theorem above.

 %To manage free nilpotent Lie algebras, our main tool is the Hall basis constructed from a set of generators. This allow us to prove that the necessary condition cannot be satisfied by any one dimensional central extension of an odd dimensional free nilpotent Lie algebra, neither a free nilpotent Lie algebra of even dimension, with the exception of % Properties of free nilpotent Lie algebras give that only one even dimensional free nilpotent Lie algebra does verify that condition, namely 
%the Lie algebras in the Theorem. Finally, we exhibit a symplectic structure for each of them. %To complete the work, we deal with one dimensional central extensions of those free nilpotent Lie algebras of odd dimensions: none of them admit symplectic structures.

\section{Free nilpotent Lie algebras}

Let $\mgg$ denote a real Lie algebra. The sequence of ideals of $\mgg$, $\{C^r(\mgg)\}$, which for non-negative integers $r$ is given by
$$
C^0(\mgg)=\mgg,\qquad C^r(\mgg)=[\mgg,C^{r-1}(\mgg)] 
$$
is called the central descending series of $\mgg$.
\smallskip

A Lie algebra $\mgg$ is called \emph{k-step nilpotent} if $C^k(\mgg)=\{0\}$ but $C^{k-1}(\mgg)\neq \{0\}$ and in this case $C^{k-1}(\mgg)\subseteq \mz(\mgg)$, where $\mz(\mgg)$ denotes the center of the Lie algebra.

A particular family of nilpotent Lie algebras is constituted by the free nilpotent ones.

\smallskip
Let $\mathfrak f_m$ denote the free Lie algebra on $m$ generators, $m\geq 2$ (notice that a unique  element spans an abelian Lie algebra). The quotient Lie algebra $\mg_{m,k}=\mathfrak f_m/C^{k+1}(\mathfrak f_m)$ is the \emph{free $k$-step nilpotent} Lie algebra on $m$ generators $\mg_{m,k}$. The image of a generator set of $\mathfrak f_m$ by the quotient map induces what is called a {\em generator set} of $\mg_{m,k}$. To each ordered set of generators $\{e_1,\ldots,e_m\}$ there is associated a basis of $\mg_{m,k}$, called a {\em Hall basis} (see \cite{Ha,GG}). Its construction is as follows.

%The \emph{free $k$-step nilpotent} Lie algebra on $m$ generators $\mg_{m,k}$ is defined as the quotient Lie algebra $\mg_{m,k}=\mathfrak f_m/C^{k+1}(\mathfrak f_m)$ where $\mathfrak f_m$ denotes the free Lie algebra on $m$ generators, with $m\geq 2$. (Notice that a unique  element spans an abelian Lie algebra). Each ordered set of generators $\{e_1,\ldots,e_m\}$ has a basis of $\mg_{m,k}$ associated that is called a {\em Hall basis} (see \cite{Ha,GG}). Its construction is as follows.

Define the \emph{length} $\ell$ of each  generator as $1$. Take the Lie brackets $[e_i, e_j]$ for $i>j$, which by definition satisfies $\ell([e_i,e_j])=2$. Now the elements $e_1, \dots , e_m$, $[e_i, e_j]$, $i>j$ belong to the Hall basis. Define a total order in that set by extending the order of the set of  generators and so that $E>F$ if $\ell(E)>\ell(F)$. They allow the construction of the elements of length $3$ and so on. 

Recursively each element of the Hall basis of $\mg_{m,k}$ is defined as follows. The generators $e_1,\ldots,e_m$ are elements of the basis  of length 1. Assume we have defined basic elements of lengths $1,\ldots, r-1\leq k-1$, with a total order satisfying  $E > F$ if $\ell(E) > \ell(F)$.

If $\ell(E) =s$ and $\ell(F) = t$ and $r = s + t\leq k$, then $[E, F]$ is a basic element of length $r$ if both of the following conditions hold:

\begin{enumerate}
\item $E$ and $F$ are basis elements and $E>F$, and
\item if $\ell(E)>1$ and $E=[G,H]$ is the unique decomposition with $G,H$ basic elements, then $F\geq H$.
\end{enumerate}

Fixed a Hall basis, denote by $\mathfrak p(m,s)$ the subspace spanned by the elements of the basis of length $s$. Hence $\mg_{m,k}$ is a graded Lie algebra since
\begin{equation}\label{e6}
[\mathfrak p(m,s),\mathfrak p(m,t)]\subseteq \mathfrak p(m,s+t)\quad \text{ and }\quad
\mg_{m,k}=\bigoplus_{s=1}^k \mathfrak p(m,s).\end{equation}

The central descending series of a free nilpotent Lie algebra verifies $$C^r(\mg_{m,k})=\oplus_{s=r+1}^k\mathfrak p(m,s).$$
This property follows from the fact that every  bracket of $r + 1$ elements of $\mg_{m,k}$, is a linear combination of brackets of $r + 1$ elements in
the Hall basis (see proof of Theorem 3.1 in \cite{Ha}) This implies $C^r(\mg_{m,k})\subseteq \oplus_{s=r+1}^k \mathfrak p(m, s)$; the
other inclusion is obvious. In particular, $\mathfrak p(m, k) = C^{k-1}(\mg_{m,k}) \subseteq \mathfrak z(\mg_{m,k})$.

\medskip

Denote as $d_m(s)$ the dimension of $\mathfrak p(m,s)$. Inductively one gets \cite{Se}
\begin{equation}\label{dim}s\cdot d_m(s)= m^s- \sum_{r|s, r<s} r\cdot d_m(r),\qquad s\geq 1.
\end{equation}

Hence for a fixed $m$, one has  $d_m(1)=m$ and $d_m(2)=m(m-1)/2$. 

\begin{example}
\label{ex:1} Let $\mg_{m,2}$ be the 2-step free nilpotent Lie algebra on $m$ generators and let $e_1,\ldots, e_m$ be a set of generators. The dimension of this Lie algebra is $d_m(1)+d_m(2)=m+m(m-1)/2$ by equations (\ref{e6}) and (\ref{dim}).

From the construction described above, a Hall basis  of $\mg_{m,2}$ is
\begin{equation}\label{ba2}
	\mathcal B=\{e_i,\,[e_j,e_k]:\,i=1,\ldots,m,\, 1\leq  k<j\leq m\}. 
	\end{equation}
	
The center of $\mg_{m,2}$ contains the first term of the central descending series $C^1(\mg_{m,2})=\mathfrak p(m,2)$, since the Lie algebra is 2-step nilpotent. Even more, $$\mathfrak z(\mg_{m,2}) = \mathfrak p(m, 2).$$
In fact, from (\ref{ba2}) above, any $x\in \mathfrak z(\mg_{m,2})$ can be written as
	$$x = \sum_{i=1}^m  x_ie_i + \sum_{m\geq i>j \geq 1} y_{ij} [e_i, e_j].$$ 	
The bracket $[x, e_1]$ is zero and equals $\sum_{i= 1}  x_i[e_i, e_1]$. This implies that 
$ x_i = 0$ for $i = 2, \hdots,m$. 
By taking the Lie bracket  $[x, e_2]$ it turns out that $x_1 = 0$ and therefore $x = 0$.
\end{example}

\begin{example}
	 \label{ex:2} For the free $3$-step nilpotent Lie algebra on $m$ generators $\mg_{m,3}$ the Hall basis associated to a set of generators has the form
	\begin{equation} \label{ba3}
\mathcal B=\{e_i,\,[e_j,e_k],\,[[e_r,e_s],e_t],\,i=1,\ldots,m,\,1\leq k< j\leq m,\,1\leq s<r\leq m,\,t\geq s\}.
	\end{equation}
	
By a similar procedure as that one in the previous example one reaches  $\mz(\mg_{m,3})=\mathfrak p(m,3)$.
Hence, by equation (\ref{dim}),
$$\dim\, \mz(\mg_{m,3})= d_m(3)=m(m^2-1)/3.$$
\end{example}

\section{Free nilpotent Lie algebras and symplectic structures}

In this section a necessary condition for a free nilpotent Lie algebra to admit a symplectic structure is proved. This condition relates the dimension of the center of $\mg_{m,k}$ with $m$, the amount of generators (see corollary \ref{cor1}). Afterwards, the proof of the Main Theorem is given. %Moreover, the dimension of the center of free nilpotent Lie algebras is computed to prove that it is not satisfied by the most part of the Lie algebras in that family.
\medskip

A Lie algebra $\mgg$ of even dimension $2n$ is called {\em symplectic} if it has a closed 2-form $\omega$ such that $\omega^n\neq 0$. Equivalently, $\omega$ as a skew symmetric bilinear form on $\mgg$ is non degenerate. 

\begin{example}\label{ex4} The one dimensional central extension of the free 2-step nilpotent Lie algebra on 2 generators, $\mg_{2,2}$, is symplectic. 

The dimension of $\mg_{2,2}$ is three (recall example \ref{ex:1}) and it is isomorphic to the Heisenberg Lie algebra. Its central extension $\mgg=\R\oplus \mg_{2,2}$ has a basis $\{e_1,e_2,e_3,e_4\}$ where the only non zero bracket is $[e_2,e_3]=e_4$. Let $\{e^i\}_{i=1}^4$ be the dual basis. Then the Maurer-Cartan formula asserts that the differential $d:\mgg^*\lra \Lambda^2\mgg^*$ behaves in that basis in the following way
$$d\,e^i=0,\;i=1,2,3, \quad d\,e^4=-e^2\wedge e^3. $$
It is easy to verify that $\omega=e^1\wedge e^2+e^3\wedge e^4$ is a symplectic structure on $\mgg$.
\end{example}

\begin{example}\label{ex:3}
%Let  $\mgg$ be the abelian Lie algebra of dimension $2n$ and $\{e^1,\ldots,e^{2n}\}$ a basis of one forms. Then the $2$-form $$\omega=e^1\wedge e^2+e^3\wedge e^4+\cdots +e^{2n-1}\wedge e^{2n} $$ is symplectic.
The $2$-step free nilpotent Lie algebra on three generators, $\mg_{3,2}$ is symplectic. 
By example \ref{ex:1}, the dimension of $\mg_{3,2}$ is six and it has a basis of the form $\{e_1,\ldots, e_6\}$  with non zero brackets $$[e_1,e_2]= e_4,\quad [e_1,e_3]= e_5,\quad [e_2,e_3]= e_6.$$

The differential of the Lie algebra in the dual basis $\{e^1,e^2,\ldots, e^6\}$ of one forms is
 $$\left\{\begin{array}{l}
de^4=-e^1\wedge e^2\\
de^5=-e^1 \wedge e^3\\
de^6=-e^2\wedge e^3\end{array}
\right.. $$

The $2$-form $\omega= e^1\wedge e^4+e^2\wedge e^6+ e^3\wedge e^5$  is closed and $\omega^3\neq 0$, hence it is a symplectic structure on $\mg_{2,3}$.

\end{example}

The follwoing necessary condition for nilpotent Lie algebras to admit  symplectic structures was given by Benson and Gordon in \cite{BE-GO}. Here we include the proof of Guan (\cite{GU}). 
%The set of left invariant symplectic forms on $N$ can be identified with
%$$\mathcal{S}(\mgg) = \left\{\sigma \in\Lambda^2\mgg^*:\; d\sigma=0\;\text{ and }\; \sigma^k\neq 0\right\}.$$

%This set, possibly empty, is an open subset of $\ker\, d :\Lambda^2\mgg^*\lra\Lambda^3\mgg^*$.

%In this section we derive the proof of the main theorem. This is done in several steps. First we give a necessary condition for a nilpotent Lie algebra to be equipped with a symplectic structure. Secondly we  prove that these conditions are sufficient.

\medskip
\begin{lem}\cite{BE-GO,GU}\label{lm1} Let $\omega$ be a symplectic structure on a nilpotent Lie algebra $\mgg$, then
\begin{equation} \label{eq:eq51} 
\dim\mz\,(\mgg)\leq \dim (\mg/C^1(\mgg)) .
\end{equation} 
\end{lem}

\begin{proof}Let $\mz(\mgg)\,^\omega=\{x\in\mgg\,/\, \omega(x,z)=0,\;\forall \,z\in\mz\,(\mgg)\}$ be the orthogonal space of the center with respect to the symplectic structure. Since $\omega$ is non degenerate it follows, $\dim\mgg=\dim\mz(\mgg)+\dim\mz(\mgg)\,^\omega$. 
Moreover,  $\dim\mgg=\dim C^1(\mgg)+\dim (\mgg/C^1(\mgg))$. We claim that
\begin{equation}\label{eq:nodeg}
\dim C^1(\mgg)\leq \dim\mz(\mgg)^\omega .
\end{equation}

In fact, let $y=[y_1,y_2]$ in $C^1(\mgg)$. Since $\omega$ is closed, we have for any $z\in\mz(\mgg)$,
$$\omega(z,y)=\omega(z,[y_1,y_2])=\omega([z,y_2],y_1)+\omega([z,y_1],y_2)=0 .$$
Hence $C^1(\mgg)\subseteq\mz(\mgg)\,^\omega$ and  equation (\ref{eq:nodeg}) holds.

Since $\omega$ is non degenerate, we have
$$\dim\mz(\mgg)+\dim\mz(\mgg)^\omega=\dim C^1(\mgg)+\dim (\mg/C^1(\mgg)).$$ 
and together with equation (\ref{eq:nodeg}) we obtain the thesis.
\end{proof} 

\begin{remark}
Condition  (\ref{eq:eq51}) is not sufficient in general. In fact, any filiform Lie algebra $\mgg$ satisfies equation (\ref{eq:eq51}) since $\dim\,\mz\,(\mgg)=1$ and $\dim(\mgg/C^1(\mgg))=2$. Nevertheless, there are filiform Lie algebras admitting no symplectic structures (see for instance, \cite{BO-GO}). 

In \cite{PO-TI}, Poussele and Tirao showed that it is sufficient for the existence of symplectic structures in the family of nilpotent Lie algebras associated to graphs.
\end{remark}

Clearly, $\dim \left(\mg_{m,k}/C^1(\mg_{m,k})\right)=m$ for  $\mg_{m,k}$ the $k$-step free nilpotent Lie algebra on $m$ generators. Hence the  equation (\ref{eq:eq51}) gives the following corollary.

\begin{cor}\label{cor1} Let $\mg_{m,k}$ be the free $k$-step nilpotent Lie algebra on $m$ generators and consider
the Lie algebra $\mgg=\R^t\oplus \mg_{m,k}$ where $t=0$ or $t=1$ depending on whether $\dim \mg_{m,k}$ is even or odd.
If $\mgg$ admits a symplectic structure then 
\begin{equation}
\label{e4}
\dim \mz\,(\mg_{m,k}) \leq m.
\end{equation}
\end{cor}

\begin{proof} If $\dim\mg_{m,k}$ is even, then $\mgg=\mg_{m,k}$ and equation (\ref{e4}) follows directly from Lemma \ref{lm1}. Let $\mgg$ be the direct sum $\R\oplus \mg_{m,k}$ where $\mg_{m,k}$ has odd dimension. %Lemma \ref{cor1} applied to $\mgg$  asserts that if it admits symplectic structures then
%\begin{equation}\label{e5}
%\dim\,\mz\,(\mgg) \leq \dim(\mgg/C^1(\mgg)) .
%\end{equation}
In this case, $\dim\mz(\mgg)=\dim\mz(\mg_{m,k})+1$. Moreover,
$\dim(\mgg/C^1(\mgg))=$ $\dim(\mg_{m,k}/C^1(\mg_{m,k}))+1$ since $C^1(\mgg)=$$C^1(\mg_{m,k})$. Hence, condition (\ref{eq:eq51}) for $\mgg$ is equivalent to (\ref{e4}).
\end{proof}

This condition on the dimension of the center is very restrictive for free nilpotent Lie algebras because, except for small $m$ and $k$, the dimension of $\mz(\mg_{m,k})$ is much bigger than the size of the generator set. %The formalization of this idea is the proof oFormalizing this idea,  the following result is reached. This is what we prove in the following proposition.

\medskip
{\bf Proof of the Main Theorem.} 
%\begin{thm}\label{p1}  Let $\mg_{m,k}$ be the free $k$-step nilpotent Lie algebra on $m$ generators. % and having even dimension.
%Then $\mg_{m,k}$ admits symplectic structures if and only if $(m,k)=(3,2)$
%\end{thm}
Let $\mg_{m,k}$ denote the free $k$-step nilpotent Lie algebra on $m$ generators. We treat separately the cases by the nilpotency index $k$ of $\mg_{m,k}$.

$\bullet$  $k=2$. As shown in example \ref{ex:1}, the dimension of the center  $\mg_{m,2}$, is $m(m-1)/2$. Easy computations give that $\dim\,\mz(\mg_{m,2})>m$ except for $m=2$ and $m=3$. Hence, by (\ref{e4}) in Corollary \ref{cor1}, $\mg_{m,2}$ or its one dimensional central extensions are not symplectic for all $m\geq 4$. 

For $m=2$, the Lie algebra $\mg_{2,2}$ has dimension three. In example \ref{ex4} it was shown that its one dimensional central extension is symplectic.
 
The case $m=3$, the Lie algebra $\mg_{3,2}$ was treated in example \ref{ex:3} and it is also symplectic.
\smallskip

$\bullet$ $k=3$. Example \ref{ex:2} asserts that $\dim\,\mz(\mg_{m,3})=m(m^2-1)/3$. Then (\ref{e4}) does not hold for $m>2$. Therefore, $\mg_{m,3}$ and $\R\oplus \mg_{m,3}$ are not symplectic if $m>2$.

Despite the fact that $\dim\,\mz(\mg_{2,3})$ equals the amount of generators, the six dimensional Lie algebra $\mgg=\R\oplus \mg_{2,3}$ does not admit symplectic structures (see, for instance \cite{BO-GO}).
\medskip

To continue, we show that for any $k\geq 4$ and any $m\geq 2$ the dimension of $\mz\,(\mg_{m,k})$ is always grater than $m$. This fact together with the previous corollary imply that neither $\mg_{m,k}$ nor its one dimensional central extension is symplectic.

\smallskip

$\bullet$ $k= 4$. The subspace $\mathfrak p(m,4)$ is contained in $\mz(\mg_{m,4})$, thus from (\ref{dim}):
$$\dim\mz(\mg_{m,4})\geq d_m(4)=\frac{1}{4}(m^4-d_m(1)-2d_m(2))= \frac{m^2(m^2-1)}{4}.$$ 
Notice that $m^2(m^2-1)/4>m$ whenever $m\geq 2$.
\smallskip

$\bullet$ $k\geq 5$. It is possible to give a lower bound of $\dim\mz(\mg_{m,k})$ by constructing different elements of length $k$ in a Hall basis $\mathcal B$ of $\mg_{m,k}$. 

Let $\{e_1,\ldots,e_m\}$ a set of generators of $\mg_{m,k}$ and consider the set
\begin{eqnarray}
\mathcal U &=&\{[[[e_i,e_j],e_k],e_m]: 1\leq j<i\leq m,\,k\geq j\}.\nonumber  %\\
% \mathcal U_2&=& \{[[e_m,e_{m-1}],[e_i,e_j]]:m\geq i>j\geq 1,\, (i,j)\neq (m,m-1)\}.\nonumber 
\end{eqnarray}

Any element in $\mathcal U$ is basic and of length 4. Given $x\in \mathcal U$, the bracket $$[x,e_m]^{(s)}:= [[[x,\overbrace{e_m],e_m]\cdots , e_m]}^{s}  \;\;s\geq 1$$
is an element in the Hall basis if $\ell([x,e_m]^{(s)})\leq k$. 

In fact if $s=1$ then $[x,e_m]^{(1)}= [x,e_m]$ and it holds:
\begin{enumerate}
\item both $x=[[[e_i,e_j],e_k],e_m] \in \mathcal U$ and $e_m$ are elements of the Hall basis, and $x >e_m$ because of their length;
\item also $x=[G,H]$ with $G=[[e_i,e_j],e_k]$ and $H=e_m$ and we have $e_m\geq H$. 
\end{enumerate}

So both conditions of the Hall basis definition are satisfied. Hence 
 $[x,e_m]^{(1)}\in \mathcal B$ and it also belongs to $ C^4(\mg_{m,k})$.
 
 Inductively suppose  $[x,e_m]^{(s-1)}\in\mathcal B$, then clearly $[[[x,e_m],e_m]\cdots ], e_m]^{(s-1)}>e_m$ and it is possible to write $[x,e_m]^{(s-1)}=[G,H]$ with $H=e_m$. Thus $[x,e_m]^{(s)}\in\mathcal B$.  Notice that $[x,e_m]^{(s)}\in C^{s+3}(\mg_{m,k})$.

We construct the following set 
$$ \widetilde{\mathcal U}:=\{[x,e_m]^{(k-4)}:x\in \mathcal U\}\subseteq C^{k-1}(\mg_{m,k}).$$
It is contained in the center of $\mg_{m,k}$ and it is a linearly independent set. Therefore
\begin{equation}\label{cota}
\dim \mz(\mg_{m,k}) \geq |\widetilde{\mathcal U}|.
\end{equation}

Clearly $\widetilde{\mathcal U}$ and $\mathcal U$ have the same cardinal. Also, $|\mathcal U|=\sum_{j=1}^m(m-j+1)(m-j)$ since for every fixed $j=1,\ldots,m$,  the amount of possibilities to choose $k\geq j$ and $i>j$ is $(m-j+1)$ and $(m-j)$  respectively. 

Straightforward computations give $ |\widetilde{\mathcal U}|=1/3\, m^3+m^2+2/3\,m$ which toghether with (\ref{cota}) proves that for any $m$ and $k\geq 5$ $$\dim \mz(\mg_{m,k})\geq 1/3 \,m^3+m^2+2/3\,m .$$ The right hand side is greater than $m$ for all $m\geq 2$.

\smallskip

\begin{flushright}$\square$\end{flushright}

\begin{remark} Our Main Theorem here extends the results in \cite[Example 4.9]{DO-TI}. In fact, they prove the non existence of symplectic structures on 2-step free nilpotent Lie algebras with different techniques. Those do not apply for every degree of nilpotency.
\end{remark}

\medskip

%The Main Theorem can be stated for the nilmanifolds associated to free nilpotent Lie algebras.

%Malcev proved in \cite{MA} that a Lie group $G$ with Lie algebra $\mathfrak{g}$ having rational structure constants admits a discrete subgroup such that the quotient $M=\Gamma\backslash G$ is a compact manifold.

Denote by $N_{m,k}$ the simply connected Lie group corresponding to $\mg_{m,k}$, the free $k$-step nilpotent Lie algebra on $m$ generators. It is well known that the structure constants of $\mg_{m,k}$ relative to a Hall basis are rational (see for instance \cite{RE}). A result due to Malcev (\cite{MA}) asserts that the Lie group $N_{m,k}$ admits a cocompact discrete subgroup $\Gamma$. Recall the correspondence between symplectic structures on the nilmanifold $M=\Gamma\backslash N_{m,k}$ and symplectic structures on $\mg_{m,k}$. Therefore, the Main Theorem can be stated in terms of the nilmanifolds $N_{m,k}$, namely:
\smallskip

{\bf Theorem} {\em Let $N_{m,k}$ be the simply connected Lie group with Lie algebra $\mg_{m,k}$ and $\Gamma $ a cocompact subgroup.
\begin{enumerate}
\item If $\dim N_{m,k}$ is even then the nilmanifold $M=\Gamma\backslash N_{m,k}$ admits symplectic structures if and only if $(m,k)=(3,2)$.
\item If $\,\dim N_{m,k}$ is odd, the nilmanifold $S^1\times \Gamma\backslash N_{m,k}$ admits symplectic structures if and only if $(m,k)=(2,2)$.
\end{enumerate}}

\medskip
 {\bf Acknowledgments.} This paper is part of the author's Ph.D. thesis, written at FCEIA, Universidad Nacional de Rosario, Argentina and directed by Isabel Dotti. The author is grateful to Isabel Dotti for all the supervision work and to Gabriela Ovando for useful comments and suggestions.

%\bibliographystyle{alpha}

%\bibliography{biblio}

\end{document}